\documentclass[12pt,twoside]{amsart}
\usepackage{amssymb}
\usepackage{graphicx}
\usepackage[margin=1in]{geometry}

\newcommand{\beas}{\begin{eqnarray*}}
\newcommand{\eeas}{\end{eqnarray*}}
\newcommand{\bea}{\begin{eqnarray}}
\newcommand{\eea}{\end{eqnarray}}
\newcommand{\beq}{\begin{equation}}
\newcommand{\eeq}{\end{equation}}
\newcommand{\ben}{\begin{enumerate}}
\newcommand{\een}{\end{enumerate}}

\newtheorem{theorem}{Theorem}[section]
\newtheorem{lemma}[theorem]{Lemma}
\newtheorem{proposition}[theorem]{Proposition}

\newtheorem{conjecture}[theorem]{Conjecture}

\theoremstyle{definition}
\usepackage{latexsym}
\usepackage{color,hyperref}
\definecolor{darkblue}{rgb}{0,0,0.6}
\hypersetup{colorlinks,breaklinks,linkcolor=darkblue,urlcolor=darkblue,anchorcolor=darkblue,citecolor=darkblue}

\title[The Catalan case of Armstrong's conjecture on core partitions]{The Catalan case of Armstrong's conjecture on simultaneous core partitions}

\author[Richard P. Stanley]{Richard P. Stanley} \address{Department of Mathematics\\ MIT\\ Cambridge, MA 02139-4307}
\email{rstan@math.mit.edu}

\author[Fabrizio Zanello]{Fabrizio Zanello} \address{Department of Mathematical  Sciences\\ Michigan
  Tech\\ Houghton, MI  49931-1295} 
\email{zanello@mtu.edu}


\begin{document}

\begin{abstract}
A beautiful recent conjecture of D. Armstrong predicts the average
size of a partition that is simultaneously an $s$-core and a $t$-core,
where $s$ and $t$ are coprime. Our goal is to prove this
conjecture when $t=s+1$. These simultaneous $(s,s+1)$-core partitions, which are enumerated by Catalan
numbers, have average size $\binom{s+1}{3}/2$. 
\end{abstract}

\maketitle

\section{Introduction and some simple cases}

Let $\lambda=(\lambda_1,\lambda_2,\dots, \lambda_m)$ be a
\emph{partition of size $n$}, i.e., the $\lambda_i$ are weakly
decreasing positive integers summing to $n$. We can represent
$\lambda$ by means of its \emph{Young} (or \emph{Ferrers})
\emph{diagram}, which consists of a collection of left-justified rows
where row $i$ contains $\lambda_i$ cells. To each of these cells $B$
one associates its \emph{hook length}, that is, the number of cells in
the Young diagram of $\lambda$ that are directly to the right or below
$B$ (including $B$ itself).  Figure~\ref{1.1fig} represents the Young
diagram of the partition $\lambda=(5,3,3,2)$ of size 13; the number
inside each cell represents its hook length.

Let $s$ be a positive integer. We say that $\lambda$ is an
\emph{$s$-core} if $\lambda$ has no hook of length equal to $s$ (or
equivalently, equal to a multiple of $s$). For instance, from
Figure~\ref{1.1fig} we can see that $\lambda=(5,3,3,2)$ is an $s$-core
for $s=6$ and for all $s\ge 9$. Finally, $\lambda$ is an
\emph{$(s,t)$-core} if it is simultaneously an $s$-core and a $t$-core.

The theory of $(s,t)$-cores has  been the focus of much interesting
research in recent years (see \cite{And, AHJ, OS} for some of the main
results). In particular, when $s$ and $t$ are coprime, there exists
only a finite number of $(s,t)$-core partitions. In fact, there are
exactly $\binom{s+t}{s}/(s+t)$ such cores (see \cite{And}), the largest
of which has  size $(s^2-1)(t^2-1)/24$ \cite{OS}. More generally, a
nice result of J. Anderson \cite{And} provides  a bijective
correspondence between $(s,t)$-cores and order ideals of the poset of the positive integers that are not contained in
the numerical semigroup generated by $s$ and $t$, which we write as $P_{(s,t)}$. The partial order on
$P_{(s,t)}$ is determined by specifying that $a$ covers $b$ whenever
$a-b$ equals either $s$ or $t$. (Our poset terminology follows \cite[Chap.~3]{ec1}.) 

\begin{figure}
\centerline{\includegraphics[width=4cm]{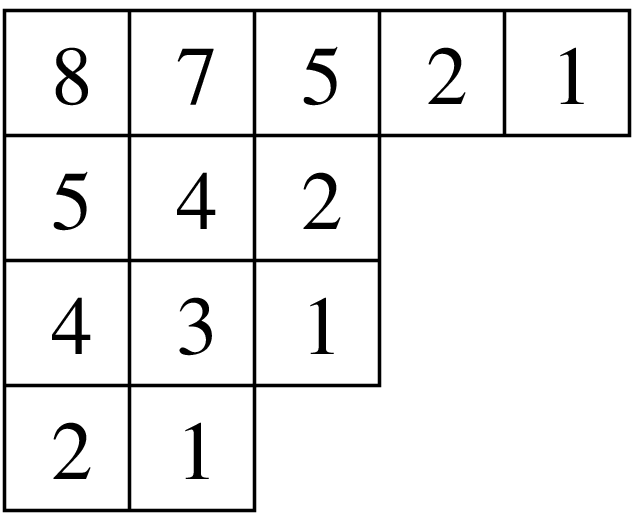}}
\caption{The Young diagram of $\lambda=(5,3,3,2)$. The number inside
  each cell indicates its hook length.} 
\label{1.1fig}
\end{figure}

For instance, let $s=3$ and $t=5$. Then $P_{(3,5)}=\{1,2,4,7\}$, where
$7>4>1$ and $7>2$. Figure~\ref{1.2fig}
represents the Hasse diagram of the poset $P_{(3,5)}$, rotated
45$^\circ$ counterclockwise from the usual convention. The order ideals of $P_{(3,5)}$ are the
following $\frac{1}{3+5}\binom{3+5}{3}=7$ subsets: $\emptyset, \{1\}, \{2\},
\{2,1\}$, $\{4,1\}$, $\{4,2,1\}$, and $\{7,4,2,1\}$.  Notice that
from this diagram it is clear that if an element $a$ of $P_{(3,5)}$
belongs to a given order ideal $I$, then all elements immediately to
the right or below $a$ also belong to $I$.

\begin{figure}
\centerline{\includegraphics[width=2cm]{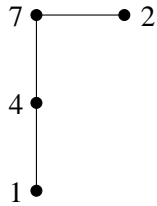}}
\caption{The Hasse diagram of the poset $P_{(3,5)}$.}
\label{1.2fig}
\end{figure}

Anderson's result then gives that $(s,t)$-cores correspond bijectively
to the order ideals of $P_{(s,t)}$ by associating  the ideal
$\{a_1,\dots, a_j\}$, where $a_1>\dots >a_j$, to the $(s,t)$-core
partition $(a_1-(j-1), a_2-(j-2), \dots, a_{j-1}-1, a_j)$. For
instance, the $(3,5)$-cores are the following seven partitions:
$\emptyset$ (corresponding to the order ideal $\emptyset$ of
$P_{(3,5)}$), $(1)$ (corresponding to $\{1\}$), $(2)$ (corresponding
to $\{2\}$), $(1,1)$ (corresponding to $\{2,1\}$), $(3,1)$
(corresponding to $\{4,1\}$), $(2,1,1)$ (corresponding to
$\{4,2,1\}$), and $(4,2,1,1)$ (corresponding to $\{7,4,2,1\}$). 

The following conjecture of D. Armstrong, informally stated sometime
in 2011 and then recently published in \cite{AHJ}, predicts, for
any $s$ and $t$ coprime, a surprisingly simple formula for the average
size of an $(s,t)$-core. 

\begin{conjecture}\label{arm}
For any coprime positive integers $s$ and $t$, the average size of an 
$(s,t)$-core is $(s+t+1)(s-1)(t-1)/24$. Equivalently, the sum of the
sizes of all $(s,t)$-cores is 
\begin{equation}\label{eee}
\frac{(s+t+1)(s-1)(t-1)}{24(s+t)}\binom{s+t}{s}.
\end{equation}
\end{conjecture}
  
For instance, the seven $(3,5)$-cores computed above are of size
$0,1,2,2,4,4,$ and $8$, with average size $3$, as predicted by
Armstrong's conjecture.  

One of the interesting aspects of this conjecture, besides the
partition-theoretic result that it predicts, is the extra combinatorial
information that it would imply on numerical semigroups generated by
two elements. In fact, even though one would generally expect these
semigroups to be very well understood, Armstrong's conjecture had
until now resisted all attempts of significant progress. 

The main goal of this paper is to show the conjecture in what is
probably its most interesting case, namely that of
$(s,s+1)$-cores. The number of these cores is the \emph{Catalan
  number} $C_s:=\frac{1}{s+1}\binom{2s}{s}$, and the corresponding posets
$P_{(s,s+1)}$ present a particularly nice structure, which will allow
us to use induction in the proof. 

We now wrap up this first section by briefly discussing Armstrong's
conjecture in a few initial cases. For any given $s$, in principle the
conjecture can be verified computationally for all $(s,t)$-cores,
given how explicitly one can determine these cores by means of
Anderson's bijection. In fact, the authors of \cite{AHJ} indicate that
C. Hanusa has verified the conjecture for small values of $s$, though
they provide no details in the paper. (We thank C. Hanusa for subsequently informing us that he had checked the conjecture on Mathematica for all $(s,ms+1)$-cores and $(s,ms-1)$-cores, when $s\le 10$.) We also wish to remark here that, since this paper has appeared as a preprint, lots of  work has already been done that has applied or extended our ideas. For instance, for a nice proof of the case $(s,ms+1)$ of Armstrong's conjecture, for arbitrary $s$ and $m\ge 1$, see \cite{Agg}, while for two interesting works on multiple simultaneous cores, see \cite{AL,YZZ}. Instead, for a complete proof of the analogous of Armstrong's conjecture for the case of self-conjugate $(s,t)$-cores (also stated in \cite{AHJ}), see \cite{CHW}.

In this section, we will just present a short proof of
the case $(3,t)$ of Armstrong's
conjecture (the case $(2,t)$ being trivial), which also gives us the
opportunity to state a simple but useful fact on arbitrary
$(s,t)$-cores that seems to have not yet been recorded in the
literature. We will provide this lemma without proof, since the
argument is analogous to the classical proof that if $\lambda$
is an $s$-core,  then it is also an $ms$-core, for all $m\ge 1$ (see
e.g. the first author's \cite{ec2}, Exercise 7.60 and its solution on
pp. 518--519).  In principle, the use of this lemma would considerably
simplify a ``brute-force'' proof for any given $s$, and indeed the
case $s=4$ is still  relatively quick to prove along the same lines;
nonetheless, for higher values of $s$ the computations remain
extremely unpleasant. 
 
\begin{lemma}\label{st}
If a partition $\lambda$ is an $(s,t)$-core, then it is also an
$(s,s+t)$-core. 
\end{lemma}

\begin{proposition}
Armstrong's conjecture holds for all $(3,t)$-cores.
\end{proposition}

\begin{proof} 
Let $s=3$. We will show Formula (\ref{eee}) for $t=3n-2$, the case $t=3n-1$ being entirely
similar. Notice that by Lemma \ref{st} all $(3,3n-2)$-cores are
$(3,3n+1)$-cores. Thus by induction, proving the result is now
equivalent to showing that the sum of the sizes of the
$(3,3n+1)$-cores that are \emph{not} also $(3,3n-2)$-cores is the
difference between the two total sums predicted by Armstrong's
conjecture, namely 
$$\Delta(n)= \frac{(3n+5)\cdot 2\cdot 3n}{24(3n+4)}\binom{3n+4}{3}-
\frac{(3n+2)\cdot 2\cdot
  (3n-3)}{24(3n+1)}\binom{3n+1}{3}=\binom{3n+2}{3}.$$ 

Figure~\ref{1.3fig} represents the Hasse diagrams of $P_{(3,10)}$ and
$P_{(3,13)}$. From these diagrams, we can see that the order
ideals of $P_{(3,13)}$ that are not also in $P_{(3,10)}$ are exactly
the six principal  ideals generated by 11, 14, 17, 20, 23, and 10,
plus the seven  ideals generated by $\{2, 10\}$, $\{5, 10\}$, $\{8,
10\}$, $\{11, 10\}$, $\{14, 10\}$, $\{17, 10\}$, and $\{20,10\}$. 

\begin{figure}
\centerline{\includegraphics[width=6cm]{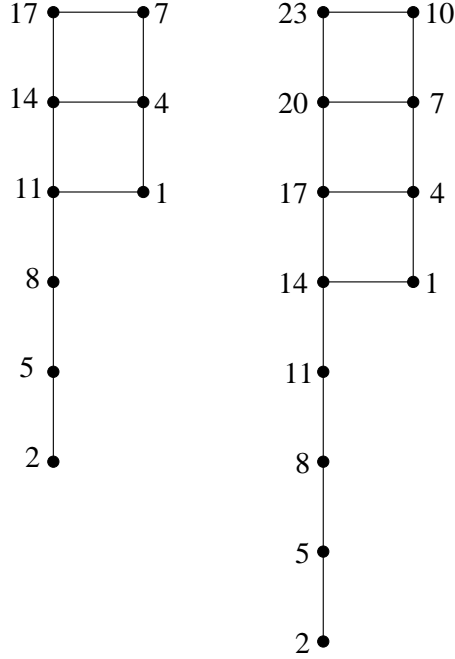}}
\caption{The Hasse diagrams of $P_{(3,10)}$ (on the left) and
  $P_{(3,13)}$.} 
\label{1.3fig}
\end{figure}

In a similar fashion, it can  be seen that the order ideals of
$P_{(3,3n+1)}$ but not of $P_{(3,3n-2)}$ are exactly the $n+2$
principal ideals generated by $3n-1$, $3n+2$, ..., $6n-1$, and $3n-2$,
and the $2n-1$ ideals generated by $\{2,3n-2\}$, $\{5,3n-2\}$, ...,
$\{6n-4,3n-2\}$.  

A standard computation now gives that $\Delta(n)$, i.e., the sum
of the elements of the above order ideals $I$ minus $\binom{\# I}{2}$,
where $\#I$ denotes the cardinality of $I$, is given by
$$\Delta(n)=(2+5+\dots  +(3n-1)) + \sum_{i=n}^{2n-1}[(2+5+\dots +(3i+2))+(1+4+\dots+3(i-n)+1)]$$$$+2n(1+4+\dots+(3n-2))+\sum_{i=0}^{2n-2}(2+5+\dots+(3i+2))$$$$-\left[\binom{n}{2}+\sum_{i=n}^{2n-1}\binom{2i-n+2}{2}+\sum_{i=0}^{2n-1}\binom{n+i}{2}\right].$$  

Showing now that the right-hand-side is equal to $\binom{3n+2}{3}$ is
a routine task that we omit. This completes the proof. 
\end{proof}

We only remark here that using Lemma \ref{st}, Armstrong's conjecture
can also be verified relatively quickly for $s=4$, i.e., for all
$(4,2n+1)$-cores (though the computations are of course already much
more tedious than for $s=3$). In fact, by  Formula (\ref{eee}), in this case one has to show
that the sum of the sizes of all $(4,2n+1)$-cores equals
$S(n):=(4n+6)\binom{n+3}{4}.$ It is easy to check (see also
\cite{A005585}) that, for all $n\ge 7$, the sequence $S(n)$ satisfies
the following curious recursive relation: 
$$\sum_{i=0}^6(-1)^{i}\binom{6}{i}S(n-i)=0.$$

It would be very interesting to combinatorially explain this identity
in the context of $(4,2n+1)$-cores, and thus give an elegant proof of
Armstrong's conjecture for $s=4$. 

\section{The Catalan case}

The goal of this section is to show Armstrong's conjecture for
$(s,s+1)$-cores. We denote by $T_s:=P_{(s,s+1)}$ the corresponding
poset according to Anderson's bijection \cite{And}. For simplicity, we
will draw the Hasse diagram of $T_s$ from top to bottom; thus, each
element of $T_s$ covers the two elements immediately below, and the
elements increase by $s$ at each step up and to the left, and by $s+1$
at each step up and to the right. (See Figure~\ref{2.1fig} for the
Hasse diagram of $T_5$.)

\begin{figure}
\centerline{\includegraphics[width=4cm]{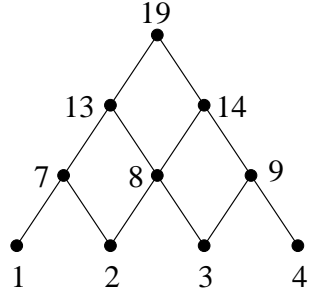}}
\caption{The Hasse diagram of $T_5$.}
\label{2.1fig}
\end{figure}

Let us define the  functions 
 \beas g_j & := & \frac{j(j-1)}{12}\binom{2j}{j},\\ 
  f_j & := & \frac{j^2+5j+2}{8j+4}\binom{2j+2}{j+1}-4^j,\\
  h_j & := & 2^{2j-1}-\binom{2j+1}{j}+\binom{2j-1}{j-1}, \eeas
where by convention we set  $h_0:=0$. We need the following two
identities. We thank Henry Cohn for verifying them for us on Maple. 

\begin{lemma}\label{wz1}
$$f_s=\sum_{i=1}^sC_{s-i}(2f_{i-1}+h_{i-1}).$$
\end{lemma}

\begin{proof}
This is the Maple code that verifies the identity (it gives 0 as output):\\
g := j -$>$ binomial(2*j,j)*j*(j-1)/12;\\
f := j -$>$ binomial(2*j+2,j+1)*(j$\hat{{\ }}$2+5*j+2)/(8*j+4)-4$\hat{{\ }}$j;\\
h := j -$>$ 2$\hat{{\ }}$(2*j-1)-binomial(2*j+1,j)+binomial(2*j-1,j-1);\\
C := j -$>$ binomial(2*j,j)/(j+1);\\
simplify(sum(C(s-i)*(2*f(i-1)+h(i-1)),i=2..s)-f(s));
\end{proof}

\begin{lemma}\label{wz2}
$$g_s=\sum_{i=1}^s2C_{s-i}g_{i-1}+2(s-i+1)C_{s-i}f_{i-1}+(s-i+3)C_{s-i}h_{i-1}+(i-1)C_{s-i}C_{i-1}-h_{s-i}h_{i-1}.$$
\end{lemma}

\begin{proof}
This is the Maple code that verifies the identity (it gives 0 as output):\\
g := j -$>$ binomial(2*j,j)*j*(j-1)/12;\\
f := j -$>$ binomial(2*j+2,j+1)*(j$\hat{{\ }}$2+5*j+2)/(8*j+4)-4$\hat{{\ }}$j;\\
h := j -$>$ 2$\hat{{\ }}$(2*j-1)-binomial(2*j+1,j)+binomial(2*j-1,j-1);\\
C := j -$>$ binomial(2*j,j)/(j+1);\\
simplify(sum(2*C(s-i)*g(i-1)+2*(s-i+1)*C(s-i)*f(i-1)+(s-i+3)*C(s-i)*h(i-1) + (i-1)*C(s-i)*C(i-1)-h(s-i)*h(i-1),i=2..s-1)
+ 2*C(0)*g(s-1)+2*(1)*C(0)*f(s-1)+(3)*C(0)*h(s-1) + (s-1)*C(0)*C(s-1)
- g(s));
\end{proof}

\begin{theorem}\label{cat}
Armstrong's conjecture holds for all $(s,s+1)$-cores.
\end{theorem}

\begin{proof}
For any given $s$, and for any weight function $w: T_s \rightarrow
{\mathbb Z}$, define the two functions 
$$f_s(w):=\sum_{I\in J(T_s)}\sum_{a\in I}w(a),$$
$$g_s(w):=\sum_{I\in J(T_s)}\left(\sum_{a\in I}w(a)-\binom{\# I}{2}\right)=f_s(w)-\sum_{I\in J(T_s)}\binom{\# I}{2},$$
where as usual $J(P)$ denotes the set of order ideals of a poset $P$.

We consider three weight functions on $T_s$. The weight $\sigma$ is
the ``standard weight'' that associates, to each element of $T_s$,
itself as a weight; i.e., $\sigma(a)=a$, for all $a\in T_s$. The
weight $\tau$ is identically 1; i.e., $\tau(a)=1$, for all $a\in
T_s$. Finally, $\rho$ records the ranks of the elements of $T_s$, when
we see this latter as a ranked poset whose minimal elements have rank
0.  Figure~\ref{2.2fig} represents the Hasse diagrams of $T_5$, where
the elements are being weighted according to $\tau$ and $\rho$. 

\begin{figure}
\centerline{\includegraphics[width=10cm]{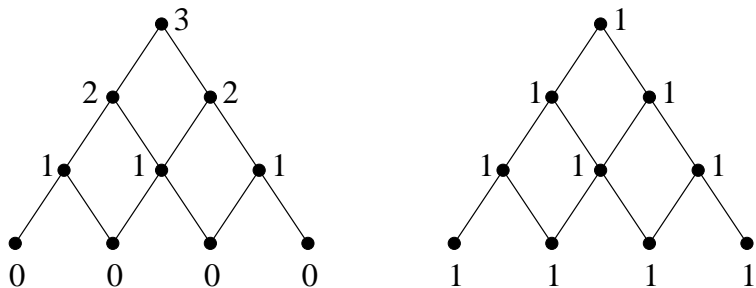}}
\caption{The poset $T_5$ with weights $\rho(a)$ on the left and
  $\tau(a)$ on the right.} 
\label{2.2fig}
\end{figure}

Showing Armstrong's conjecture for $(s,s+1)$-cores in the form of Formula (\ref{eee}) is
tantamount to proving that 
$$ g_s(\sigma)=g_s=\frac{s(s-1)}{12}\binom{2s}{s}. $$

Notice that the elements of rank 0 of $T_s$ are $1, 2, \dots,s-1$. We
can partition $J(T_s)$ as $J(T_s)=\bigcup_i J_i(T_s)$, where $J_i(T_s)$
is the set of those order ideals of $T_s$ whose least element that
they do \emph{not} contain is $i$. Notice that either $1\le i\le s-1$,
or we are considering order ideals whose least missing element $i$ (if
any) has positive rank. With some abuse of notation, in this latter
case we set by convention $i:=s$, so that  we can write
$$J(T_s)=\bigcup_{i=1}^sJ_i(T_s).$$

Notice that, given $i$, the elements $I$ of $J_i(T_s)$ \emph{must}
contain all of $1, 2, \dots ,i-1$, \emph{cannot} contain any element
covering $i$ (this is an empty condition for $i=s$), and \emph{may or
  may not} contain any other element.  Figure~\ref{2.3fig} gives the
Hasse diagram of $T_{10}$; for $i=5$, it indicates by squares the elements
of $T_{10}$ that must belong to any given order ideal $I\in
J_5(T_{10})$, by open circles the elements that cannot be in $I$, and by solid circles the elements that may or may not be in $I$.

\begin{figure}
\centerline{\includegraphics[width=12cm]{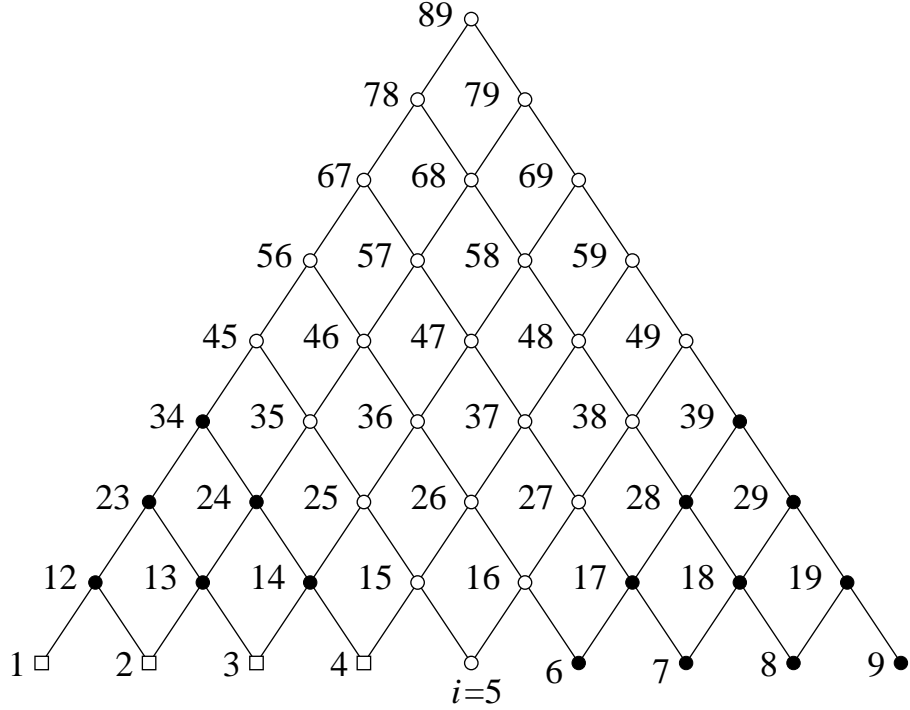}}
\caption{The possible elements of the order ideals $I\in J_5(T_{10})$. Elements that must appear in $I$ are indicated by
  squares, that cannot appear by open circles, and that may or may not  appear by solid circles.}
\label{2.3fig}
\end{figure}

It follows that any given order ideal $I\in J_i(T_{s})$ can be
partitioned into the disjoint union of two order ideals, say $I_1$ and
$I_2$, plus the elements $1, 2, \dots, i-1$. Notice that, in the Hasse
diagram of $T_s$, $I_1$ belongs to a poset that is isomorphic to
$T_{i-1}$ and sits to the left of $i$ (starting in rank one), and
$I_2$  belongs to a poset that is isomorphic to $T_{s-i}$ and sits to
the right of $i$. The posets $T_1$ and $T_0$, which arise when $i=1$, $i=2$, $i=s-1$, or $i=s$, are empty. (See again
Figure~\ref{2.3fig} for the case $n=10$ and $i=5$.)

Given this, it is a simple exercise to show that the sum of the
elements of $T_s$ that belong to a given order ideal $I=I_1\cup I_2
\cup \{1,2,\dots, i-1\}\in J(T_s)$ is given by:
$$\sum_{a\in I} \sigma(a)=\sum_{a\in I_1} w(a)+ \sum_{a\in I_2}w(a)+ \binom{i}{2},$$
where the weight function $w$ is defined as
$$w:= \sigma + (s+1)\tau + (s-i+1)\rho$$
over $I_1$, and by
$$w:= \sigma + i\tau + i\rho$$
over $I_2$. Further, notice that, given $i$, we can choose the
order ideals $I_1\in J(T_{i-1})$ and $I_2\in J(T_{s-i})$
independently. Therefore, the elements $a\in I_1$ will appear a total
of $C_{s-i}$ times in the order ideals $I$ of $T_s$, and similarly,
the elements $a\in I_2$ will appear a total of $C_{i-1}$ times in the
order ideals $I$ of $T_s$. 

Therefore, the contribution of any given
$i$ to the desired function $g_s(\sigma)$ is given by

\begin{equation}\label{iii}
m(i,s) -\sum_{I_1\in J(T_{i-1}), I_2\in J(T_{s-i})}\binom{\# I_1+\# I_2   +i-1}{2},
\end{equation}
where we have 

\begin{center}
$$m(i,s):=\sum_{I_1\in J(T_{i-1})}C_{s-i}\left(\sum_{a\in
  I_1}w(a)+\binom{i}{2}\right) + \sum_{I_2\in
  J(T_{s-i})}C_{i-1}\sum_{a\in I_2}w(a)$$$$=C_{s-i}(f_{i-1}(\sigma)+(s+1)f_{i-1}(\tau)+
(s-i+1)f_{i-1}(\rho))$$$$+C_{s-i}C_{i-1}\binom{i}{2} +C_{i-1}(f_{s-i}(\sigma)+if_{s-i}(\tau)+if_{s-i}(\rho)).$$ 
\end{center}

Let us now consider, again for a fixed $i$, the term that is being subtracted in Formula (\ref{iii}). Notice that
$$\binom{\# I_1+\# I_2 +i-1}{2}= \binom{\# I_1}{2}+\binom{\#
  I_2}{2}+(i-1)\# I_1 +(i-1)\# I_2+(\# I_1)(\# I_2)+\binom{i-1}{2}.$$ 
  
Thus, once we sum over all $I_1$ and $I_2$, similar considerations to
the above on the number of such order ideals give us that 

\begin{center}
$$ \sum_{I_1\in J(T_{i-1}), I_2\in J(T_{s-i})}\binom{\# I_1+\# I_2
  +i-1}{2}= \sum_{I_1\in J(T_{i-1})}C_{s-i}\left(\binom{\#
  I_1}{2}+(i-1)\# I_1\right) $$
  $$+\sum_{I_2\in
  J(T_{s-i})}C_{i-1}\left(\binom{\# I_2}{2}+(i-1)\#
I_2\right)+\left(\sum_{I_1\in J(T_{i-1})}\#
I_1\right)\left(\sum_{I_2\in J(T_{s-i})}\#
I_2\right)+C_{s-i}C_{i-1}\binom{i-1}{2}.$$
\end{center}

Essentially by definition, we have $\sum_{I_1\in J(T_{i-1})}\#
I_1=f_{i-1}(\tau)$, and likewise, $\sum_{I_2\in J(T_{s-i})}\#
I_2=f_{s-i}(\tau)$. Also, it is a known fact (see e.g. \cite{A006419}) that  the function $f_j(\tau)$ appearing in the above formula for $m(i,s)$ satisfies 
$$ f_j(\tau)=2^{2j-1}-\binom{2j+1}{j}+\binom{2j-1}{j-1}. $$

As for determining $f_j(\rho)$,  by employing the above decomposition
of the order ideals $I$ and summing over all $i$, with a similar
argument we can see that:
$$ f_s(\rho)= \sum_{i=1}^s C_{s-i}(f_{i-1}(\rho)+f_{i-1}(\tau))+C_{i-1}f_{s-i}(\rho),$$ 
which, by rearranging the indices, yields:
$$f_s(\rho)=\sum_{i=1}^s C_{s-i}(2f_{i-1}(\rho)+f_{i-1}(\tau)).$$

Therefore, by induction, if we apply Lemma \ref{wz1} with
$f_j=f_j(\rho)$ and $h_j=f_j(\tau)$, we promptly get the following
formula for $f_j(\rho)$: 
$$f_j(\rho)=\frac{j^2+5j+2}{8j+4}\binom{2j+2}{j+1}-4^j.$$

Finally, notice that $g_{i-1}(\sigma)=f_{i-1}(\sigma)-\sum_{I_1\in
  J(T_{i-1})}\binom{\# I_1}{2},$  and similarly for $g_{s-i}(\sigma)$. 

Therefore, by Formula (\ref{iii}) and the subsequent formula for $m(i,s)$, if we sum over  $i=1,2, \dots, s$,
after some tedious but routine computations (that include
rearranging the indices where necessary)  we obtain: 
\begin{center}
 $$ g_s(\sigma)= \sum_{i=1}^s
 [2C_{s-i}g_{i-1}(\sigma)+2(s-i+1)C_{s-i}f_{i-1}(\rho)$$
 $$+(s-i+3)C_{s-i}f_{i-1}(\tau)+(i-1)C_{s-i}C_{i-1}-f_{i-1}(\tau)f_{s-i}(\tau)].$$ 
\end{center}

The theorem now follows by induction on $s$, if we apply Lemma
\ref{wz2} with $f_j=f_j(\rho)$, $g_j=g_j(\sigma)$, and
$h_j=f_j(\tau)$. 
\end{proof}

%We conclude by recording the following curious identity, which can be
%deduced from Theorem \ref{cat}. We will omit the proof, but the
%interested reader may want to verify that, in fact, the
%right-hand-side is the sum, over all order ideals $I$ of $T_s$,  of
%the elements of $I$ minus $\binom{\# I}{2}$. 
%
%\begin{proposition}
%For any integer $s\ge 1$, we have:
%
%$$s(s-1)\binom{2s}{s}/12=\sum \sum_{c=1}^k \sum_{\alpha=i_c}^{\min
%  \{s-j_c,i_{c-1}-1\}}s\binom{s-\alpha-j_c+1}{2}+\binom{s-j_c+1}{2}-\binom{\alp%ha}{2}$$
% $$-\binom{\sum_{c=1}^k
%  \sum_{\alpha=i_c}^{\min \{s-j_c,i_{c-1}-1\}}s-\alpha-j_c+1}{2},$$ 
%where $i_0:=s$, and the external sum is taken over all $k=1,\dots
%,s-1$ and all  possible pairs $(i_1,j_1),\dots,(i_k,j_k)$  such that
%$1\le j_1<j_2< \dots <j_k\le s-1\ge i_1>i_2>\dots >i_k\ge 1$, where
%$i_h+j_h\le s$ for all $h$. 
%\end{proposition}

\section{Acknowledgements} The second author warmly thanks the first for his hospitality during calendar year 2013. The two authors wish to
thank the referee for a careful reading of their manuscript and several comments that improved the presentation, and Henry Cohn for verifying the identities of Lemmas \ref{wz1} and \ref{wz2} on the computer package Maple. The present work was done while the first author was partially supported by the National Science Foundation under Grant No.~DMS-1068625, and the second author by a Simons Foundation grant (\#274577). 

%-------------------------------------------

\end{document}